\documentclass{amsart}
\usepackage{graphicx} 
\usepackage[utf8]{inputenc}
\usepackage{amsmath, amsthm}
\usepackage{amssymb}
\usepackage{todonotes}
\usepackage{comment}
\usepackage{tikz-cd}
\newtheorem{theorem}{Theorem}
\newtheorem{theorem*}{Theorem}
\newtheorem*{definition}{Definition}
\newtheorem{lemma}[theorem]{Lemma}
\newtheorem{prop}[theorem]{Proposition}

\newtheorem{claim}[theorem]{Claim}

\newcommand{\col}{\mathrm{Col}}

\newcommand{\add}{\mathrm{Add}}
\newcommand{\la}{\langle}
\newcommand{\ra}{\rangle}

\newcommand{\shi}{\text{\usefont{U}{min}{m}{n}\symbol{'127}}}

\DeclareFontFamily{U}{min}{}
\DeclareFontShape{U}{min}{m}{n}{<-> udmj30}{}

\newcommand{\bbA}{\mathbb{A}}
\newcommand{\bbB}{\mathbb{B}}

\newcommand{\bbP}{\mathbb{P}}
\newcommand{\bbQ}{\mathbb{Q}}
\newcommand{\bbR}{\mathbb{R}}
\newcommand{\bbS}{\mathbb{S}}

\newcommand{\bbU}{\mathbb{U}}

\newcommand{\calB}{\mathcal{B}}

\newcommand{\calG}{\ensuremath{\mathcal{G}}}

\newcommand{\p}{\mathcal{P}}

\DeclareMathOperator{\dom}{dom}

\title{Comparing forcing approaches to dense ideals}
\author{Monroe Eskew}
\address{Kurt G\"odel Research Center,
University of Vienna,
Vienna, Austria.}
\email{monroe.eskew@univie.ac.at}
\thanks{This work was supported by the Austrian Science Fund (FWF) through Projects P34603 and PIN1355423.}

\begin{document}
\begin{abstract}
    We analyze some posets involved in forcing constructions for dense ideals, showing that the Anonymous Collapse and the Dual Shioya Collapse are equivalent for collapsing a large cardinal to $\omega_2$.  We also give a somewhat simplified construction of a normal ideal $I$ on $\omega_2$ such that $\p(\omega_2)/I \sim \col(\omega_1,\omega_2)$.
\end{abstract}
\maketitle

\section{Introduction}
For a regular cardinal $\kappa$, a $\kappa$-complete ideal $I$ on $\kappa$ is said to be \emph{saturated} if $\p(\kappa)/I$ has the $\kappa^+$-c.c., which is the smallest possible chain condition when $\kappa$ is a successor cardinal.
The study of saturated ideals on successor cardinals was initiated by Kunen \cite{kunen}, who proved that it is consistent relative to a huge cardinal that there is a saturated ideal on $\omega_1$.  If $I$ is a saturated ideal on a successor cardinal $\kappa$, and $G \subseteq \p(\kappa)/I$ is generic, then one can form the generic ultrapower embedding $j : V \to M \subseteq V[G]$, where $\kappa$ is the critical point of the elementary embedding $j$, and $M$ is a transitive ${<}j(\kappa)$-closed subclass of $V[G]$.  Thus it is natural to say that $\kappa$ is ``generically almost-huge'', since it satisfies the definition of almost-hugeness except that the target model is constructed in a generic extension.  For this reason, Kunen \cite{kunen} suggested that the consistency strength of the existence of such ideals might be that of an almost-huge cardinal.  This was refuted by Foreman-Magidor-Shelah \cite{fms}, who showed that one can force the existence of a saturated ideal on $\omega_1$ from a supercompact cardinal.  This was later reduced by Shelah-Woodin \cite{woodinshelah} to a Woodin cardinal, and Jensen-Steel \cite{js} proved that a Woodin cardinal is necessary.

However, at the time of this writing, the upper bound for the consistency strength of saturated ideals on successor cardinals above $\omega_1$ has still not been reduced below an almost-huge cardinal.  Magidor modified Kunen's approach to reduce the bound from huge to almost-huge (see \cite[Section 7.11]{foremanhandbook}).  Several variations on Kunen's forcing, most notably by Laver \cite{laver}, Foreman \cite{moresat}, and Shioya \cite{shioyaeaston}, have been presented, which obtain some combination of stronger combinatorial properties of the ideal, simultaneous occurrence of saturated ideals on many cardinals, and a simplification of the forcing.

A kind of ``ultimate'' saturation property is density.  We say a $\kappa$-complete ideal $I$ on a successor cardinal $\kappa$ is \emph{$\kappa$-dense}, or simply \emph{dense} for brevity, if $\p(\kappa)/I$ has a dense subset of size $\kappa$.  Most of the known variations on Kunen's approach do not obtain this property.  In unpublished work in the late 70s, Woodin gave a forcing argument for obtaining a dense ideal on $\omega_1$ from an almost-huge cardinal.  Although it shared some general features with Kunen-style constructions, it differed in a few important ways, one of which was the involvement of a choiceless model as an intermediate step in the argument.  A version of this argument was presented by the author in his thesis \cite{thesis} and in \cite{eskewdense}, in a way that generalizes to get dense ideals on arbitrary successors of regular cardinals, as well as analogous ideals over $\p_\kappa(\lambda)$.

For dense ideals on $\omega_1$, quite different approaches have been developed using weaker large cardinal assumptions.  Woodin \cite{woodinbook} showed that $\mathrm{NS}_{\omega_1}$ being dense is equiconsistent with $\mathsf{ZF+AD}$, which is equiconsistent with $\mathsf{ZFC}$ plus infinitely many Woodin cardinals.  For one direction, Woodin forced over a model of $\mathsf{ZF+AD}$ $+$ $V = L(\bbR)$, obtaining a model of $\mathsf{ZFC}$ $+$  ``$\mathrm{NS}_{\omega_1}$ is dense'' $+$ $2^\omega = \omega_2$.  By Shelah \cite{shelahdense}, $\mathrm{NS}_{\omega_1}$ being dense requires $\neg\mathsf{CH}$.
To obtain a dense ideal on $\omega_1$ along with $\mathsf{CH}$, Woodin forced over a model of $\mathsf{ZF}$ $+$ $\mathsf{AD}_\bbR$ $+$ ``$\theta$ is regular'' with a different forcing.  This was later shown to be optimal by Adolf-Sargsyan-Trang-Wilson-Zeman \cite{sargsyan}.  Recently, Lietz \cite{lietz} showed how to force $\mathrm{NS}_{\omega_1}$ to be dense from a model of $\mathsf{ZFC}$ with an inaccessible limit of supercompact cardinals, answering a question of Woodin.  Like in the case of saturation, all of these arguments are specific to $\omega_1$.

Regarding dense ideals on higher cardinals, a big mystery discussed in \cite{eskewdense} was how to get dense ideals on adjacent cardinals simultaneously, for example at $\omega_1$ and $\omega_2$.  The problem is that the density property is so sensitive that it doesn't seem to leave much room to vary the forcing (in contrast to Kunen's approach), and the forcing lacked a certain uniformity that would be helpful for iterating it to obtain dense ideals at several cardinals.

The first progress in a long time on this problem came in a recent paper of Shioya \cite{shioyadense}.  He replaced the abstract arguments about complete Boolean algebras of \cite{eskewdense} with a more structured partial order, and he eliminated the use of choiceless models.  Sensing the promise of this approach, the author worked with Hayut on a modification, adding a bit more combinatorial structure, that ultimately resulted in the solution to the problem of dense ideals on adjacent cardinals \cite{eh}.  Indeed, we obtain a model of $\mathsf{ZFC}$ in which for every regular cardinal $\kappa$, there is a $\kappa^+$-complete ideal $I$ on $\kappa^+$ such that $\p(\kappa)/I$ has a dense set isomorphic to the simple Levy collapse $\col(\kappa,\kappa^+)$.  This answered several questions of Foreman \cite{foremanhandbook}.

The purpose of this note is to compare the main forcings of  \cite{eskewdense} and \cite{eh}, which are respectively called the Anonymous Collapse $\bbA(\mu,\kappa)$ and the Dual Shioya Collapse $\shi(\mu,\kappa)$.  For technical reasons, $\shi(\mu,\kappa)$ is well-defined only when $\mu$ is a regular uncountable cardinal, while $\bbA(\mu,\kappa)$ makes sense for $\mu = \omega$.  In each case, a projection is defined from $\col(\mu,{<}\kappa) * \dot\add(\kappa)$ to the respective forcing.  $\bbA(\mu,\kappa)$ has a certain minimality property that ensures it is a subforcing of $\shi(\mu,\kappa)$.  It is natural to ask whether these two forcings are actually equivalent.  We show:


\begin{theorem}
\label{specialmu}
For inaccessible $\kappa$, $\bbA(\omega_1,\kappa)$ and $\shi(\omega_1,\kappa)$ are forcing-equivalent.
\end{theorem}

In the construction of a dense ideal using $\bbA(\mu,\kappa)$, the quotient forcing between $\bbA(\mu,\kappa)$ and $\col(\mu,{<}\kappa) * \dot\add(\kappa)$ plays a role.  This is somewhat unfortunate, as the combinatorial structure of this quotient is unclear.  In \cite{eh}, this issue is avoided through the use of a certain ``uniformization'' forcing $\bbU$, which extracts a useful subset of the generic for $\shi(\mu,\kappa)$.  It turns out that the analysis behind Theorem~\ref{specialmu} leads naturally to an elimination of these devices in an alternative argument for a dense ideal on $\omega_2$, which we present here.

\begin{theorem}
\label{omega2}
    Suppose $\kappa$ is almost-huge with target $\lambda$.  Then $\shi(\omega_1,\kappa) * \dot\col(\kappa,{<}\lambda)$ forces that there is a normal ideal $I$ on $\omega_2$ such that $\p(\omega_2)/I$ is forcing-equivalent to $\col(\omega_1,\omega_2)$.
\end{theorem}

However, we don't know how to get dense ideals on $\omega_1$ and $\omega_2$ simultaneously without the uniformization forcing $\bbU$.

\section{Canonicity of collapsing with $\sigma$-strategically closed posets}

In this section, we present a generalization of a result of Foreman \cite{foremangames}.  It is a well-known result of McAloon (see \cite[Section 4, Theorem 1]{Grigorieff}) that for every regular cardinal $\kappa$, every $\kappa$-closed poset $\bbP$ that collapses its own size to $\kappa$ is forcing-equivalent to the canonical Levy collapse $\col(\kappa,|\bbP|)$.  For $\kappa = \omega_1$, we can weaken the closure assumption to strategic closure.  This will be useful because, unlike in the case of countable closure, $\sigma$-strategic closure is always inherited by subforcings.

For a poset $\bbP$ and an ordinal $\delta$, we define two games $\mathcal{G}^\mathrm{I}_\delta(\bbP)$ and $\mathcal{G}^\mathrm{II}_\delta(\bbP)$.  Two players alternate playing elements of $\bbP$ in a descending sequence, with Player I making the first move.  At limit stages, Player I plays first in $\mathcal{G}^\mathrm{I}_\delta(\bbP)$, and Player II plays first in $\mathcal{G}^\mathrm{II}_\delta(\bbP)$.  Player II wins if the game lasts for $\delta$-many rounds; otherwise, Player I wins.  For a cardinal $\kappa$, we say that $\bbP$ is \emph{$\kappa$-strategically closed} if Player II has a winning strategy for $\mathcal{G}^\mathrm{II}_\kappa(\bbP)$, and \emph{strongly $\kappa$-strategically closed} if Player II has a winning strategy for $\mathcal{G}^\mathrm{I}_\kappa(\bbP)$.  For $\alpha = \omega+1$, Player II has a winning strategy in either game iff they have a strategy for the first $\omega$ moves that guarantees the existence of a lower bound to the play.  In this case, we say that $\bbP$ is \emph{$\sigma$-strategically closed}.

\begin{theorem}[Ishiu-Yoshinobu \cite{iy}]
\label{iythm}
A poset is $\sigma$-strategically closed iff it is strongly $\omega_1$-strategically closed.
\end{theorem}

Let us briefly outline the argument for the nontrivial direction, using:

\begin{lemma}[Ishiu-Yoshinobu \cite{iy}]
There exists a partial order $\prec$ on $\omega_1$ with the following properties:
\begin{enumerate}
    \item If $\alpha\prec\beta$, then $\alpha<\beta$.
    \item Each $\alpha<\omega_1$, the set of $\beta \prec \alpha$ is finite and linearly ordered by $\prec$.
    \item For each limit ordinal $\alpha<\omega_1$, there is a sequence $\beta_0 \prec \beta_1 \prec \beta_2 \prec \dots$ such that $\alpha = \sup_i \beta_i$.
\end{enumerate}
\end{lemma}

Suppose $\tau_0$ is a winning strategy for II in $\calG^{\mathrm{I}}_{\omega+1}(\bbP)$.  We convert this into a winning strategy $\tau_1$ for II in $\calG^{\mathrm{I}}_{\omega_1}(\bbP)$ inductively as follows.  
Suppose that we have defined $\tau_1$ for plays in which I has played ${<}\alpha$-many times.  Assume that for each $\beta<\alpha$, if 
$\la p_0,q_0,\dots,p_\beta,q_\beta \ra$ is a run of the game where II follows $\tau_1$, and $\beta_0 \prec \dots \prec \beta_{n-1} \prec \beta_n = \beta$ are the $\prec$-predecessors of $\beta$, then 
$\la p_{\beta_0},q_{\beta_0},p_{\beta_1},q_{\beta_1}\dots,p_{\beta_n},q_{\beta_n} \ra$ follows the strategy $\tau_0$.  
If $\alpha = \alpha' + 1$, then define 
$$\tau_1(\la p_0,q_0,\dots,p_{\alpha'},q_{\alpha'},p_\alpha\ra) = \tau_0(\la p_{\alpha_0},q_{\alpha_0},\dots,p_{\alpha_{n-1}},q_{\alpha_{n-1}},p_\alpha\ra),$$
where $\alpha_0 \prec \dots \prec \alpha_{n-1}$ are the $\prec$-predecessors of $\alpha$.  If $\alpha$ is a limit, then there is a sequence $\beta_0 \prec \beta_1 \prec \beta_2 \prec \dots$ such that $\alpha = \sup_i \beta_i$.  By induction, the subsequence of the play
$\la p_{\beta_0},q_{\beta_0},p_{\beta_1},q_{\beta_1},\dots \ra$ follows $\tau_0$.  Thus there is a lower bound $p_\alpha$ that Player I can play.  For any such play, define II's next move according to $\tau_1$ just as in the successor case.  This completes the construction of $\tau_1$.

To get the desired conclusion, we use the following generalization of McAloon's result:

\begin{lemma}
\label{strataloon}
    If $\bbP$ is strongly $\kappa$-strategically closed and nowhere $|\bbP|$-c.c., and $\bbP$ collapses $|\bbP|$ to $\kappa$, then $\bbP$ is forcing-equivalent to $\col(\kappa,|\bbP|)$.
\end{lemma}

\begin{proof}
    Let $\lambda =|\bbP|$.
    Let $\dot f$ be a $\bbP$-name for a surjection from $\kappa$ to the generic filter $\dot G$ with $\dot f(0) = 1_\bbP$.  We build a tree $T \subseteq \bbP$ of height $\kappa$ with the following properties:
    \begin{enumerate}
        \item $1_\bbP$ is the root of $T$.
        \item For each $t \in T$, there is a set $S(t) \subseteq T$ of size $\lambda$ such that for each $s \in S(t)$ $s < t$, and there is no $x \in T$ such that $s<x<t$.
        \item  For each descending chain $\la t_i : i < \delta \ra$ in $T$, where $\delta<\kappa$, there is a set of $\lambda$-many pairwise incompatible lower bounds in $T$ to $\la t_i: i <\delta \ra$, all of which are maximal in $T$ among such lower bounds.
        \item $T$ is dense in $\bbP$.
    \end{enumerate}
    This suffices to build a dense embedding from $\{ q \in \col(\kappa,\lambda) : \dom q$ is a successor ordinal$\}$ to $\bbP$.  We build $T$ inductively by levels, $\la T_\alpha : \alpha<\kappa \ra$, where $T_\alpha$ is the set of all $t \in T$ such that the chain in $T$ above $t$ has length $\alpha$.

    Fix a winning strategy $\tau$ for II in $\calG^{\mathrm{I}}_\kappa(\bbP)$.  We inductively assume that:
    \begin{enumerate}
        \item Each $T_\alpha$ is a maximal antichain in $\bbP$.
        \item For each $t \in T_\alpha$, if $\la t_\beta : 0 < \beta \leq \alpha \ra$ enumerates the branch above $t = t_\alpha$, then there is a run of the game
        $\la s_0,t_0,s_1,t_1,\dots,s_\alpha,t_\alpha\ra$
        following $\tau$, with $s_0 = 1_\bbP$.  We assume that these  runs are chosen so that, for $\beta<\alpha$, $t_\beta \in T_\beta$ and $t_\alpha \in T_\alpha$ with $t_\beta > t_\alpha$, the chosen run of the game above $t_\beta$ is an initial segment of that above $t_\alpha$.
        \item For each $t \in T_\alpha$, there is $p \in \bbP$ such that $t \leq p$ and $t \Vdash \dot f(\alpha) = p$.
    \end{enumerate}
Given $t_\alpha \in T_\alpha$ and the associated run of the game $\la s_0,t_0,s_1,t_1,\dots,s_\alpha,t_\alpha = t \ra$ as above, there is a dense set of $p < t$ that are the last play by II in a run of the game following $\tau$ of the form $\la s_0,t_0,s_1,t_1,\dots,s_\alpha,t_\alpha,q,p\ra$, and which have the property that for some $r \in \bbP$, $p \leq r$ and $p \Vdash \dot f(\alpha+1) = r$.  We pick a maximal antichain $A_t$ of such $p<t$, with $|A_t| = \lambda$, and let $T_{\alpha+1} = \bigcup_{t \in T_\alpha} A_t$.  

Suppose we have constructed up to a limit $\alpha$. 
Let $\vec b = \la t_\beta : \beta < \alpha \ra$ be a branch through $\bigcup_{\beta<\alpha} T_\beta$, with $t_\beta \in T_\beta$.  There is a run of the game $\la s_0,t_0,\dots,s_\beta,t_\beta,\dots \ra$ of length $\alpha$, where II plays according to $\tau$.  Thus there are lower bounds to the sequence.  Among all such lower bounds, there is a dense set of $p$ that are the last play by II in a run of the game following $\tau$ of the form $\la s_0,t_0,\dots,s_\beta,t_\beta,\dots,q,p\ra$, extending the above sequence by two points, and which have the property that for some $r \in \bbP$, $p \leq r$ and $p \Vdash \dot f(\alpha) = r$.  Let $A_{\vec b}$ be a maximal antichain of such $p$ of size $|\lambda|$, and let $T_\alpha$ be the union of the $A_{\vec b}$, over all branches $\vec b$ through $\bigcup_{\beta<\alpha} T_\beta$.
 Since $\bbP$ is $\kappa$-distributive, the set of all $p \in \bbP$ that are below some $t \in T_\beta$ for each $\beta<\alpha$ is dense, and thus $T_\alpha$ is a maximal antichain in $\bbP$.

Let $T = \bigcup_{\alpha<\kappa} T_\alpha$.  It is clear that $T$ is $\kappa$-closed.  To show that it is dense in $\bbP$, let $p_0 \in \bbP$, and let $p_1 \leq p_0$ and $\alpha<\kappa$ be such that $p_1 \Vdash \dot f(\alpha) = p_0$.  There is some $t \in T_\alpha$ that is compatible with $p_1$, and there is some $r \in \bbP$ such that $t \leq r$ and $t \Vdash \dot f(\alpha) = r$.  Thus $t \leq r = p_0$.
\end{proof}

\section{Comparing $\bbA(\omega_1,\kappa)$ and $\shi(\omega_1,\kappa)$}

Let us first recall the definitions of the forcings $\bbA(\mu,\kappa)$ and $\shi(\mu,\kappa)$.  Suppose $\kappa$ is inaccessible and $\mu<\kappa$ is regular and uncountable.

Let $\dot X$ be the canonical name for the Cohen subset of $\kappa$ added by the second stage of $\col(\mu,{<}\kappa)*\dot\add(\kappa)$.  Letting $\bbB$ be the Boolean completion of this forcing, we define $\bbA(\mu,\kappa)$ as the complete subalgebra of $\bbB$ generated by this name, i.e.\ the smallest complete subalgebra containing all Boolean values $|| \check \alpha \in \dot X ||$ for $\alpha<\kappa$.  If $G * X \subseteq \col(\mu,{<}\kappa)*\dot\add(\kappa)$ is generic over $V$, and $H$ is the induced filter on $\bbA(\mu,\kappa)$, then $V[H] = V[X]$ (see \cite[Lemma 15.40]{Jech}).  It is easy to see that all bounded subsets of $\kappa$ added by $G$ are coded in intervals of $X$, and moreover $V[H]$ and $V[G]$ have the same ${<}\kappa$-sequences of ordinals.

Let us recall a few preliminary notions before introducing the definition of $\shi(\mu,\kappa)$.  We call a poset \emph{completely $\mu$-closed} if every descending sequence of length ${<}\mu$ has a greatest lower bound.  A subset of a completely $\mu$-closed poset is called \emph{$\mu$-closed} when it is closed under descending ${<}\mu$-sequences.  A map from one $\mu$-closed poset to another is called \emph{$\mu$-continuous} when it preserves limits of descending sequences of length ${<}\mu$.  A crucial but easy fact is:
\begin{lemma}
    The intersection of ${<}\kappa$-many dense $\kappa$-closed subsets of a completely $\kappa$-closed poset is dense and $\kappa$-closed.
\end{lemma}

\begin{definition}
$\shi(\mu,\kappa)$ consists of all triples of the form $\langle \mathbb{P}, p, \tau\rangle$ where:
\begin{enumerate}
    \item $\mathbb{P}\in V_\kappa$ is a separative completely $\mu$-closed poset.
    \item $p\in \mathbb{P}$.
    \item For some $\gamma<\kappa$, $\tau$ is a $\mathbb{P}$-name for a function from $\gamma$ to $2$.
\end{enumerate}
We put $\langle \mathbb{Q}, q,\sigma\rangle \leq \langle \mathbb{P}, p,\tau\rangle$ when there is a $\mu$-continuous projection $\pi$ from a $\mu$-closed dense subset of $\mathbb{Q}$ to $\mathbb{P}$ such that $q \Vdash p \in \pi(\dot{G})$, where $\dot{G}$ is the canonical name for the $\bbQ$-generic filter, and $q \Vdash \pi^*(\tau) \trianglelefteq \sigma$, where $\pi^*$ is the canonical translation of $\bbP$-names into $\bbQ$-names via the projection $\pi$.
\end{definition}

The key feature of both $\bbA(\mu,\kappa)$ and $\shi(\mu,\kappa)$ is that they have a kind of universal property with respect to a large class of ``ordinary'' collapsing posets.

\begin{definition}
A poset $\bbR$ is called a \emph{reasonable $(\mu,\kappa)$-collapse} when:
\begin{enumerate}
    \item $\bbR$ is $\kappa$-c.c.\ and completely $\mu$-closed.
    \item There is a $\subseteq$-increasing sequence of regular suborders $\la \bbR_\alpha : \alpha<\kappa \ra$, with $\bbR = \bigcup_\alpha \bbR_\alpha$.
    \item There is a sequence of maps $\la \pi_\alpha : \alpha < \kappa \ra$ such that for all $\alpha$, $\pi_\alpha : \bbR \to \bbR_\alpha$ is a $\mu$-continuous projection, $\pi_\alpha \restriction \bbR_\alpha = \mathrm{id}$, and for $\alpha<\beta$, $\pi_\alpha = \pi_\alpha \circ \pi_\beta$.
    \item For unboundedly many regular $\alpha<\kappa$, $|\bbR_{\alpha}| = \alpha$ and $\Vdash_{\bbR_{\alpha}} |\alpha| = \mu$.
\end{enumerate}
\end{definition}

\begin{lemma}
\label{absorption}
    Suppose $\bbR$ is a $(\mu,\kappa)$-reasonable collapse.  
    \begin{enumerate}
    \item Let $\bbA(\bbR)$ be the complete subalgebra of $\calB(\bbR * \dot\add(\kappa))$ generated by the canonical name for the generic subset of $\kappa$ added by the second step.  Then $\bbA(\bbR) \cong \bbA(\mu,\kappa)$.
    \item There is a dense subset $D \subseteq \bbR * \dot\add(\kappa)$ and a projection $\psi : D \to \shi(\mu,\kappa)$.
    \item There is an $\bbA(\bbR)$-name $\dot X_\bbA$ and a $\shi(\mu,\kappa)$ name $\dot X_\shi$, both for subsets of $\kappa$, such that whenever $G * X \subseteq \bbR * \dot\add(\kappa)$ is generic, and $H$ and $K$ are the induced filters on $\bbA(\bbR)$ and $\shi(\mu,\kappa)$ respectively, then $X = \dot X_\bbA^{H} = \dot X_\shi^{K}$.
    \end{enumerate}
\end{lemma}

For a proof of (1), see \cite[Theorem 2.12]{eskewdense}; for (2), see \cite[Lemma 25]{eh}, and for (3), see \cite[Lemma 28]{eh}.  We note that $\dot X_\shi$ is defined as $\{ \alpha < \kappa : (\exists \la \bbP,p,\tau \ra \in \dot G) p \Vdash_\bbP \tau(\alpha) = 1 \}$, where $\dot G$ names the $\shi(\mu,\kappa)$-generic filter.  (3) implies that there is a complete Boolean algebra $\bbS \cong \calB(\shi(\mu,\kappa))$ such that $\bbA(\bbR) \subseteq \bbS \subseteq \calB(\bbR*\dot\add(\kappa))$.

The map $\psi$ above is defined as follows.  Fixing some sequence $\la\bbR_\alpha : \alpha < \kappa \ra$ witnessing that $\bbR$ is $(\mu,\kappa)$-reasonable, we let $D$ be the dense set of conditions $\la r,\tau \ra$ such that for some $\alpha$, $\tau$ is forced to have domain $\alpha$ and $r \in \bbR_\alpha$.  For such $\la r,\tau \ra$, we define $\psi(r,\tau) = \la \bbR_\alpha,r,\tau\ra$.  For such conditions, we have that $r \Vdash \alpha \in \dot X_{\add(\kappa)}$ iff $r \Vdash \tau(\alpha) = 1$ iff $\psi(r,\tau) \Vdash \alpha \in \dot X_\shi$.

An important notion in the analysis of $\shi(\mu,\kappa)$ is that of a \emph{coding condition}.

\begin{definition}
    We say $\la \bbP,p,\tau \ra \in \shi(\mu,\kappa)$ is a \emph{coding condition} when $\tau$ is forced to have domain $\delta<\kappa$, and whenever $G \subseteq \bbP$ is generic,
    \begin{itemize}
        \item $V[G] = V[\tau^G]$;
        \item in all generic extensions of $V[G]$, $G$ is the unique filter $F$ that is $\bbP$-generic over $V$ such that $\tau^F = \tau^G$.
    \end{itemize}
  We say that $\la \bbP,p,\tau \ra$ is a \emph{strong coding condition} when $\tau$ is forced to have domain $\delta<\kappa$, and there is a 
  dense $D \subseteq \bbP$, a set $X \subseteq \delta$, and a function $f : X \to D$ such that for all $\alpha \in X$,
    $$1_\bbP \Vdash f(\alpha) \in \dot G \leftrightarrow \tau(\alpha) = 1.$$
\end{definition}

It is easy to show that strong coding conditions are coding, and that they are dense in $\shi(\mu,\kappa)$.  Any condition can be extended to a coding condition by only lengthening the $\tau$-part.
The main fact about coding conditions is the ``Projection Freezing Lemma'':

\begin{lemma}
\label{freeze}
    Let $\langle \mathbb{P}, p, \tau\rangle$ be a coding condition and let $\langle \mathbb{Q},q,\sigma\rangle \leq \langle \mathbb{P}, p, \tau\rangle$.   For any dense $D \subseteq \bbQ$, if $\pi_0,\pi_1$ are projections from $D$ to $\bbP$ such that $q \Vdash \pi_i^*(\tau) \trianglelefteq \sigma$ for $i=0,1$, then $\pi_0(r) = \pi_1(r)$ for all $r \leq q$ in $D$.
\end{lemma}

Applying this lemma, we get that whenever $G \subseteq \shi(\mu,\kappa)$ is generic, and $\la\bbP,p,\tau\ra \in G$ is a coding condition, then we can define a local filter $G_\bbP^\tau$ as $\{ q \in \bbP : \la\bbP,q,\tau\ra \in G\}$, which will be $\bbP$-generic over $V$.  (For non-coding conditions, this definition may not yield a filter.)

Let us begin the argument for Theorem~\ref{specialmu}. Per the discussion after Lemma~\ref{absorption}, Let $\bbS \subseteq \bbB = \calB(\col(\mu,{<}\kappa)*\dot\add(\kappa))$ be the complete subalgebra corresponding to $\shi(\mu,\kappa)$, and let $\pi_\bbA,\pi_\bbS$ be the canonical projections from $\bbB$ to $\bbA(\mu,\kappa),\bbS$ respectively, so that $\pi_\bbA = \pi_\bbA \circ \pi_\bbS$.
Let $H \subseteq \bbA(\mu,\kappa)$ be a generic filter over $V$, and let $X = \dot X_\bbA^H \subseteq\kappa$ be the corresponding set of ordinals that generates $H$.  Let $\shi(\mu,\kappa)/H$ be the quotient forcing via $\pi_\bbA \restriction \bbS$.
\begin{claim}
\label{equivcond}

There is a dense set of $\la\bbP,p,\tau\ra\in\shi(\mu,\kappa)/H$ for which there exists $\alpha<\kappa$ with the following properties:
\begin{enumerate}
    \item $\alpha$ is regular in $V$ and of cardinality $\mu$ in $V[X\cap\alpha]$.
    \item $V[X]$ is a $\kappa$-strategically closed forcing extension of $V[X \cap \alpha]$.
    \item\label{shiquotient}$\la\bbP,p,\tau\ra$ is a coding condition such that $\Vdash_\bbP \dom \tau = \alpha$.
\end{enumerate}
Furthermore, every two conditions $\la\bbP,p,\tau\ra,\la\bbQ,q,\sigma\ra$ as in (\ref{shiquotient}) are equivalent in the separative quotient of $\shi(\mu,\kappa)/H$.
\end{claim}

\begin{proof}
    Let $a \in \bbA(\mu,\kappa)$ and $\la\bbQ,q,\sigma\ra\in \shi(\mu,\kappa)$ be arbitrary such that $a \Vdash \la\bbQ,q,\sigma\ra \in \shi(\mu,\kappa)/\dot H$, so that if $s \in \bbS$ corresponds to $\la\bbQ,q,\sigma\ra$, then $a \leq \pi_\bbA(s)$.  We can find $\la p,\tau \ra \in \col(\mu,{<}\kappa)*\dot\add(\kappa)$ such that:
    \begin{enumerate}
        \item $\pi_\bbA(p,\tau) \leq a$.
        \item There is $\alpha<\kappa$ such that:
        \begin{enumerate}
            \item  $\alpha$ is regular.
            \item $p \in \col(\mu,{\leq}\alpha)$.
            \item $\tau$ is a $\col(\mu,{\leq}\alpha)$-name forced to be a condition of length $\alpha$.
        \end{enumerate}
        \item $\psi(p,\tau)$ is a coding condition below $\la\bbQ,q,\sigma\ra$.
    \end{enumerate} 
 Let $G * X \subseteq \col(\mu,{<}\kappa)*\dot\add(\kappa)$ be a generic possessing $\la p,\tau\ra$.  Let $G' = G \restriction(\alpha+1)$.  Then $V[X \cap \alpha] = V[G']$, and $V[G*X]$ is a $\mu$-closed forcing extension of $V[G']$.  Since $\mu$-strategic closure is inherited by subforcings, $V[X]$ is a $\mu$-strategically closed extension of $V[X \cap \alpha]$.  If $H \subseteq \bbA(\mu,\kappa)$ and $K \subseteq \shi(\mu,\kappa)$ are the induced generic filters, then $\psi(p,\tau) \in K$, and thus $\la\col(\mu,{\leq}\alpha),p,\tau\ra \in \shi(\mu,\kappa)/H$, and it extends $\la\bbQ,q,\sigma\ra$.  By the arbitrariness of $a$ and $\la\bbQ,q,\sigma\ra$, the set of conditions satisfying properties (1)--(3) with respect to some $\alpha$ is forced to be dense.

    For the last claim, suppose $\la\bbP,p,\tau\ra \in \shi(\mu,\kappa)/H$ is coding and of length $\alpha$.  Then it is possible to force further to obtain a generic $G \subseteq \shi(\mu,\kappa)$ such that $\dot X_\shi^G = X$.  Since $\la\bbP,p,\tau\ra$ is coding, there is in $V[X\cap\alpha]$ a unique filter $F \subseteq \bbP$ that is generic over $V$ and such that $\tau^F$ is the characteristic function of $X \cap \alpha$; namely $F = G^\tau_\bbP$.  Likewise, if $\la\bbQ,q,\sigma\ra \in \shi(\mu,\kappa)/H$ is coding and of length $\alpha$, then $G^\sigma_\bbQ$ can be defined in $V[X\cap\alpha]$.  There is a $\bbP$-name $\dot g_\bbQ$ for a filter on $\bbQ$ and some $p_0 \leq p$ in $G^\tau_\bbP$ such that $p_0 \Vdash_\bbP$ ``$\dot g_\bbQ$ is the unique $V$-generic filter $F$ on $\bbQ$ such that $\sigma^F = \tau^{\dot G_\bbP}$.''
    Likewise, there is $\bbQ$-name $\dot g_\bbP$ for a filter on $\bbP$ and some $q_0 \leq q$ in $G^\sigma_\bbQ$ such that $q_0 \Vdash_\bbQ$ ``$\dot g_\bbP$ is the unique $V$-generic filter $F$ on $\bbP$ such that $\tau^F = \sigma^{\dot G_\bbQ}$.''
    Note that $(\dot g_\bbQ)^{G_\bbP^\tau} = G^\sigma_\bbQ$, and $(\dot g_\bbP)^{G_\bbQ^\sigma} = G^\tau_\bbP$.  Let $p_1 \leq p_0$ be in $G^\tau_\bbP$ and force that whenever $g \subseteq \bbP$ is generic, and $h = (\dot g_\bbQ)^g$, then $g = (\dot g_\bbP)^h$.
    Likewise, let $q_1 \leq q_0$ be in $G^\sigma_\bbQ$ and force that whenever $h \subseteq \bbQ$ is generic, and $g = (\dot g_\bbP)^h$, then $h = (\dot g_\bbQ)^g$.

    There is $p_2 \leq p_1$ in $G^\tau_\bbP$ such that for all $p' \leq p_2$, $|| p' \in \dot g_\bbP || \wedge q_1 \not= 0$, since otherwise, the set of $p' \leq p_1$ such that $q_1 \Vdash p' \notin \dot g_\bbP$ would be dense, contradicting the forced genericity of $\dot g_\bbP$.  Thus $e : p' \mapsto || p' \in \dot g_\bbP || \wedge q_1$ is a complete embedding of $\bbP \restriction p_2$ into $\calB(\bbQ) \restriction q^*$, where $q^* =  || p_2 \in \dot g_\bbP || \wedge q_1$.  Furthermore, the range of $e$ is dense in the codomain, since if $q'\leq q^*$, then we can take $h \subseteq \bbQ$ generic with $q' \in h$, and there will be some $p' \leq p_2$ that forces $q' \in \dot g_\bbQ$.  Thus if $h' \subseteq \bbQ$ is any generic with $e(p') \in h'$, then $p' \in g = (\dot g_\bbQ)^{h'}$, and $q' \in (\dot g_\bbQ)^g = h'$.  Thus $e(p') \leq q'$.
    
    

    If $e^*(\sigma)$ is the translation of $\sigma$ into a $\bbP$-name via $e$, we want to show that $p_2 \Vdash e^*(\sigma) = \tau$.  Let $g \subseteq \bbP$ be generic with $p_2 \in g$, and let $h \subseteq \bbQ$ be the generic generated by $e[g]$.  Then $g = (\dot g_\bbP)^h$ by the definition of $e$, and thus $h = (\dot g_\bbQ)^g$.  It follows that $\tau^g = \sigma^h = e^*(\sigma)^g$.
    
    
    $e$ lifts to an isomorphism $\varphi : \calB(\bbP \restriction p_2) \to \calB(\bbQ) \restriction q^*$.  Note that $\varphi$ is continuous, and $\bbP$ and $\bbQ$ appear as $\mu$-closed dense subsets of their respective Boolean completions.  The set $D = (\bbP \restriction p_2) \cap \varphi^{-1}[\bbQ \restriction q^*]$ is a dense $\mu$-closed subset of $\bbP \restriction p_2$.  Thus $\varphi \restriction D$ witnesses that $\la\bbP,p_2,\tau\ra \leq \la\bbQ,q,\sigma\ra$ in $\shi(\mu,\kappa)$.  Since $p_2 \in G^\tau_\bbP$, $\la\bbP,p_2,\tau\ra \in G$ for any generic $G \subseteq \shi(\mu,\kappa)$ with $\dot X^G_\shi = X$.  Therefore, whenever we take a generic $G \subseteq \shi(\mu,\kappa)/H$ with $\la\bbP,p,\tau\ra \in G$, there will be some $p'\in\bbP$ such that $\la\bbP,p',\tau\ra\in G$ and $\la\bbP,p',\tau\ra\leq\la\bbQ,q,\sigma\ra$, showing the desired claim.
    \end{proof}

    \begin{claim}
        Let $H \subseteq \bbA(\omega_1,\kappa)$ be generic over $V$.  Then every two conditions in the dense subset of $\shi(\omega_1,\kappa)/H$ from Claim~\ref{equivcond} are compatible in $\shi(\omega_1,\kappa)/H$.
    \end{claim}

This claim suffices to prove Theorem~\ref{specialmu}, since it implies that the quotient forcing between $\bbA(\omega_1,\kappa)$ and $\shi(\omega_1,\kappa)$ is always trivial.

\begin{proof}
    Let $H \subseteq \bbA(\omega_1,\kappa)$ be generic, and let $\la\bbP,p_0,\tau\ra,\la\bbQ,_0,\sigma\ra$ be in the dense subset of $\shi(\omega_1,\kappa)/H$ from Claim~\ref{equivcond}, with respective witnessing ordinals $\alpha\leq\beta$.  Let $X \subseteq\kappa$ be the set of ordinals that generates $H$.  As in the proof of Claim~\ref{equivcond}, there are filters $G_\bbP^\tau \subseteq \bbP$ and $G_\bbQ^\sigma\subseteq\bbQ$ definable in $V[X \cap \alpha],V[X\cap\beta]$ respectively, with the property that $V[X\cap\alpha]=V[G_\bbP^\tau]$ and $V[X\cap\beta] = V[G_\bbQ^\sigma]$.  These filters are the unique generic filters over $V$ that evaluate the names $\tau,\sigma$ to the characteristic functions of $X\cap \alpha,X\cap\beta$ respectively.

    Let $q_1\leq q_0$ be in $G_\bbQ^\sigma$ and force that $\dot g_\bbP$ is the unique $\bbP$-generic filter over $V$ that evaluates $\tau$ to $\sigma \restriction \alpha$.  Note that $(\dot g_\bbP)^{G_\bbQ^\sigma} = G_\bbP^\tau$.  As in the proof of Claim~\ref{equivcond}, there must be some $p_1 \leq p_0$ in $G_\bbP^\tau$ such that for all $p \leq p_1$, $q_1 \wedge || p \in \dot g_\bbP || \not= 0$
    The map $p \mapsto q_1 \wedge || p \in \dot g_\bbP ||$ is a complete embedding of $\bbP \restriction p_1$ into $\calB(\bbQ) \restriction (q_1 \wedge || p_1 \in \dot g_\bbP ||)$.

    Since $V[X]$ is an $\omega_1$-strategically closed extension of $V[X \cap \alpha]$, and $V[X \cap \beta]$ is an intermediate model, $V[X \cap \beta]$ is an $\omega_1$-strategically closed extension of $V[X \cap \alpha]$.  There is $q_2 \leq q_1$ in $G^\sigma_\bbQ$ that forces this, and there is $p_2 \leq p_1$ in $G^\tau_\bbP$ such that $e : p \mapsto q_2 \wedge || p \in \dot g_\bbP ||$ is a complete embedding of $\bbP \restriction p_2$ into $\calB(\bbQ) \restriction (q_2 \wedge || p_2 \in \dot g_\bbP ||)$.  
    By Theorem~\ref{iythm} and Lemma~\ref{strataloon}, 
    $$\calB(\bbQ) \restriction (q_2 \wedge || p_2 \in \dot g_\bbP ||) \cong \calB(\bbP \restriction p_2 * \dot\col(\omega_1,\beta)).$$
    It follows that there is a countably closed dense $D \subseteq \bbQ$ and a countably continuous projection $\pi : D \to \bbP$, with the property that $q_2 \wedge || p_2 \in \dot g_\bbP ||$ forces that $\pi[\dot G_\bbQ]$ generates $\dot g_\bbP$.  If $q_3 \leq q_2 \wedge || p_2 \in \dot g_\bbP ||$ is in $G^\sigma_\bbQ$, then $\la \bbQ,q_3,\sigma \ra \leq \la \bbP,p_0,\tau\ra$, and $\la \bbQ,q_3,\sigma \ra \in \shi(\omega_1,\kappa)/H$.
\end{proof}

We remark that, for general $\shi(\mu,\kappa)$, the above argument can be carried out up to the point of getting the complete embedding from $\bbP$ to $\calB(\bbQ)$ (restricted to some conditions) with $\mu$-strategically closed quotient.  But perhaps it is not strongly $\mu$-strategically closed, and so Lemma~\ref{strataloon} is not applicable.  

Let $D \subseteq \shi(\mu,\kappa)/H$ be the dense set from Claim~\ref{equivcond}, and let $f : D \to \kappa$ give the witnessing ordinal $\alpha$ to membership in $D$.  For any generic $G \subseteq \shi(\mu,\kappa)/H$, let $Y_G = f[D\cap G]$.  Then $Y_G$ is some unbounded subset of $\kappa$, and $V[H][Y_G] = V[G]$, by Claim~\ref{equivcond}.  When $\mu = \omega_1$, we have that $Y_G = f[D]$, but perhaps for $\mu>\omega_1$, $Y_G \notin V[H]$.

\section{An alternative argument for a dense ideal on $\omega_2$}

Theorem~\ref{iythm} and Lemma~\ref{strataloon} imply that $\bbA(\omega_1,\kappa) \times \col(\omega_1,\kappa)$ is forcing-equivalent to $\col(\omega_1,\kappa)$.  It follows that, for inaccessible $\lambda>\kappa$, $[\bbA(\omega_1,\kappa) \times \col(\omega_1,\kappa)] * \dot\col(\kappa,{<}\lambda)^{\bbA(\omega_1,\kappa)}$ is equivalent to a reasonable $(\omega_1,\lambda)$-collapse.  This observation leads to an adjustment of the proof of Theroem 2.20 in \cite{eskewdense} to show that, if $\kappa$ is almost-huge with target $\lambda$, then in $V^{\bbA(\omega_1,\kappa)*\dot\col(\kappa,{<}\lambda)}$, we can lift the almost-hugeness embedding by forcing with $\col(\omega_1,\kappa)$, without the need for a generic for the quotient between $\bbA(\omega_1,\kappa)$ and $\col(\omega_1,{<}\kappa)*\dot\add(\kappa)$.  Choiceless submodels still make some appearance in the lifting procedure.  We conclude that in $V^{\bbA(\omega_1,\kappa)*\dot\col(\kappa,{<}\lambda)}$, there is a normal ideal $I$ on $\omega_2 =\kappa$ such that $\p(\kappa)/I \cong \calB(\col(\omega_1,\omega_2))$.

We would like to present here a direct argument for the existence of such an ideal on $\omega_2$ after forcing with $\shi(\omega_1,\kappa) * \dot\col(\kappa,{<}\lambda)$, using the idea of Lemma~\ref{strataloon} as a key component.  This will avoid the use of choiceless submodels.  Rather than appeal to Theorem~\ref{iythm} abstractly, we will directly build a strategy witnessing the strong $\omega_1$-strategic closure of $\shi(\omega_1,\kappa)$ that will have some additional useful properties.

Let us recall the notion of an inverse limit of partial orders under a system of projections.
Suppose $\vec\bbP = \la \bbP_\alpha : \alpha < \delta \ra$ is a sequence of posets and $\vec\pi = \la \pi_{\beta\alpha} : \alpha<\beta<\delta \ra$ is a sequence such that $\pi_{\beta\alpha} : \bbP_\beta \to \bbP_\alpha$ is a projection.  Suppose also that whenever $\alpha<\beta<\gamma<\delta$, then $\pi_{\gamma\alpha} = \pi_{\beta\alpha}\circ\pi_{\gamma\beta}$.   Such a pair of sequences is called an \emph{inverse system}, and $\delta$ is referred to as the \emph{length} of the system.  We define the \emph{inverse limit} of the system, $\varprojlim(\vec\bbP,\vec\pi)$, to be the set of all sequences $\la p_i : i <\delta \ra$ such that for $\alpha<\beta<\delta$, $p_\alpha = \pi_{\beta\alpha}(p_\beta)$.  The ordering on $\varprojlim(\vec\bbP,\vec\pi)$ is pointwise.

We also allow inverse systems on which the maps $\pi_{\beta\alpha}$ are not necessarily defined at every point of $\bbP_\beta$, and the maps are not assumed to commute everywhere.
In this case, we take $\varprojlim(\vec\bbP,\vec\pi)$, to be the set of all sequences $\la p_i : i <\delta \ra$ such that for all $\alpha<\beta<\delta$, $p_\beta \in \dom \pi_{\beta\alpha}$ and $p_\alpha = \pi_{\beta\alpha}(p_\beta)$.  

Lemma 15 of \cite{eh} shows that, if we have a commuting inverse system $\la\vec\bbP,\vec\pi\ra$ of length $\omega$, where each $\bbP_i$ is assumed to be completely $\omega_1$-closed, and the projections $\pi_{ji}$ are $\omega_1$-continuous and defined on $\omega_1$-closed dense sets, then the inverse limit is a completely $\omega_1$-closed poset that comes with a natural system of commuting $\omega_1$-continuous projections to each $\bbP_i$.  For inverse systems of uncountable length, it seems that we need to assume that this behavior already occurs at many limit points in order to draw an analogous conclusion.

Let us now define our strategy for Player II in $\calG^{\mathrm{I}}_{\omega_1}(\shi(\omega_1,\kappa))$.  We require II to play coding conditions at each round.  We will also carry a record of projections between the posets of II's plays that witness their descending ordering, and we require each of these maps to be defined \emph{everywhere}.  This last requirement is important because, per Remark 34 in \cite{eh}, failure to do this might produce a descending $\omega$-sequence of conditions in $\shi(\omega_1,\kappa)$ without a lower bound.

Assume $\la a_0,b_0,a_1,b_1,\dots \ra$ is a run of the game of length $\alpha<\omega_1$ where II plays according to a strategy that satisfies the above requirements.  So each $b_i$ is a coding condition of the form $\la\bbP_i,p_i,\tau_i\ra$, and for $i<j<\alpha$, we have an $\omega_1$-continuous projection $\pi_{ji} : \bbP_j \to \bbP_i$.  First assume $\alpha = \beta+1$, let $b_{\beta} = \la\bbP,p,\tau\ra$ be the last move of II, and let $\la\bbQ,q,\sigma\ra$ be the next play of I.  First choose a coding condition $\la\bbP',p',\tau'\ra\leq\la\bbQ,q,\sigma\ra$, and let $\pi'$ witness the ordering $\la\bbP',p',\tau'\ra\leq\la\bbP,p,\tau\ra$.  Then let $\bbP'' = \bbP' \cap \dom \pi'$, which is an $\omega_1$-closed dense suborder, and let $\pi_{\beta\alpha} = \pi' \restriction \bbP''$.  For $\gamma<\beta$, let $\pi_{\alpha\gamma} = \pi_{\beta\gamma}\circ\pi_{\alpha\beta}$.  Let $p'' \in \bbP''$ be below $p'$, and let II's next play be $\la\bbP'',p'',\tau'\ra$.

Next suppose $\alpha$ is a limit.  Let $\la\alpha_i : i < \omega \ra$ be an increasing sequence converging to $\alpha$.  Let $\la\la\bbP_{\alpha_i},p_{\alpha_i},\tau_{\alpha_i}\ra : i < \omega  \ra$ list the plays of II at rounds $\alpha_i$.  Since each play of II is coding, we have that whenever $i < j <k < \omega$ and $p \leq p_{\alpha_k}$, then $\pi_{\alpha_k,\alpha_i}(p) = \pi_{\alpha_j,\alpha_i}\circ\pi_{\alpha_k,\alpha_j}(p)$.  For each $i < \omega$, let $p^*_{\alpha_i} = \inf \{ \pi_{\alpha_j,\alpha_i}(p_{\alpha_j}) : i < j < \omega \}$.
Let $\bbP$ be the inverse limit of the system $\la\la \bbP_{\alpha_i} : i < \omega \ra,\la\pi_{\alpha_j,\alpha_i} : i < j < \omega\ra\ra$.  We have that $\la p^*_{\alpha_i} : i < \omega \ra \in \bbP$, and below this condition, the system of restriction maps $p \mapsto p(i)$ is a commuting system of $\omega_1$-continuous projections from $\bbP$ to the $\bbP_{\alpha_i}$'s.  Therefore, a lower bound to the sequence of plays exists.

Player I responds with some lower bound $\la\bbQ,q,\sigma\ra$ at round $\alpha$.  As explained in \cite[footnote 7]{eh}, $\bbQ$ does \emph{not} have to project to the inverse limit of the previous rounds.  But whatever I plays, II can play a coding condition $\la\bbP,p,\tau\ra$ below it.  II then chooses a system of projections $\la \pi'_{\alpha\beta} : \beta<\alpha \ra$ from $\bbP$ to the posets of II's previous plays witnessing that $\la\bbP,p,\tau\ra$ is below them.  II then takes $\bbP_\alpha$ to be $\bigcap_{\beta<\alpha} \dom \pi'_{\alpha\beta}$, which is a dense $\omega_1$-closed subset of $\bbP$, and puts $\pi_{\alpha\beta} = \pi'_{\alpha\beta} \restriction \bbP_\alpha$.  II plays $\la\bbP_\alpha,p_\alpha,\tau\ra$, where $p_\alpha \in \bbP_\alpha$ is below $p$.

This completes the construction of the strategy for II; call it $\Sigma$.
Now, the proof of Lemma~\ref{strataloon} gives a way of building a dense copy of $\col(\omega_1,\kappa)$ in $\shi(\omega_1,\kappa) \times\col(\omega_1,\kappa)$ using $\Sigma$.  We use it to construct a certain master condition for $\shi(\omega_1,\kappa)$ in $\shi(\omega_1,\lambda)$, for inaccessible $\lambda>\kappa$.  This is similar to Lemma 43 of \cite{eh}, but simpler.

Let $\Sigma'$ be a winning strategy for II in $\calG^{\mathrm{I}}_{\omega_1}(\shi(\omega_1,\kappa)\times\col(\omega_1,\kappa))$, where the first coordinate follows $\Sigma$ and the second coordinate just plays below the previous moves arbitrarily (which works by countable closure).  Using the construction of Lemma \ref{strataloon}, build a dense tree $T \subseteq \shi(\omega_1,\kappa)\times\col(\omega_1,\kappa)$ that is isomorphic to the set of $p \in \col(\omega_1,\kappa)$ whose domain is a successor ordinal.  By the construction, every (upward-closed) branch through the tree comes along with a run of the game where II plays according to $\Sigma'$, and the plays of II are the nodes of the tree that are in the branch.

Since the posets appearing in $\shi(\omega_1,\lambda)$ need to be completely $\omega_1$-closed $T$ is not quite suitable.  But we can correct this by artificially adding infima to every bounded branch of $T$, obtaining $T' \supseteq T$ that is isomorphic to the set of conditions $p \in \col(\omega_1,\kappa)$ whose domain is an ordinal.

As in the previous section, let $\dot X_\shi$ be the name for the special subset of $\kappa$ added by $\shi(\omega_1,\kappa)$.  We interpret this as a $T'$-name using the factorization above.  Let $\tau_X$ be a $T'$-name for the characteristic function of $\dot X_\shi$.  Let $\la\bbP,p,\tau\ra$ be a node of $T$.
\begin{claim}
\label{mc}
    For inaccessible $\lambda >\kappa$, $\la T',\la\bbP,p,\tau\ra,\tau_X \ra$ is below $\la\bbP,p,\tau\ra$ in $\shi(\omega_1,\lambda)$.
\end{claim}
\begin{proof}
    For nodes $\la\bbQ,q,\sigma\ra \in T$ below $\la\bbP,p,\tau\ra$, the strategy $\Sigma$ carries an $\omega_1$-continuous projection $\pi : \bbQ \to \bbP$ witnessing the ordering.  We define a map $\rho : T' \restriction \la\bbP,p,\tau\ra$ first on the successor nodes, where at such $\la\bbQ,q,\sigma\ra \in T$, we put $\rho(\la\bbQ,q,\sigma\ra) = \pi(q)$ where $\pi : \bbQ \to \bbP$ is the projection chosen by $\Sigma$.  For limit nodes $t \in T'$ below $\la\bbP,p,\tau\ra$, first select a sequence $\la\la\bbQ_i,q_i,\sigma_i\ra : i <\omega\ra$ of nodes of $T$ converging to $t$.  Let $\pi_i : \bbQ_i \to \bbP$ be the projection chosen by $\Sigma$.  Since the nodes of $T$ are coding, these projections commute below the conditions $q_i$, and $\la \pi_i(q_i) : i < \omega \ra$ is a descending sequence in $\bbP$.  We define $\rho(t) = \inf_i \pi_i(q_i)$.  Again because the nodes in $T$ are coding, it does not matter which $\omega$-sequence converging to $t$ we pick, since any two of them will project to interleaved descending sequences in $\bbP$.

    By construction, $\rho$ is an $\omega_1$-continuous map.  To show it is a projection, first suppose $\la\bbQ,q,\sigma\ra$ is a successor node below $\la\bbP,p,\tau\ra$, and $p' \leq \pi(q)$, where $\pi$ is the projection chosen by $\Sigma$.  There is $q' \in \bbQ$ below $q$, and since $\la\bbQ,q',\sigma\ra \leq \la\bbQ,q,\sigma\ra$ and $T$ is dense, there is $\la\bbR,r,\chi\ra \in T$ below $\la\bbQ,q',\sigma\ra$.  If $\pi' : \bbR \to \bbQ$ and $\pi'' : \bbR \to \bbP$ are the projections chosen by $\Sigma$, then $\pi''(r) = \pi \circ\pi'(r) \leq \pi(q') \leq p'$.  If $t$ is a limit node below $\la\bbP,p,\tau\ra$, and $p' \leq \rho(r)$, then there is an $\omega$-sequence $\la\la\bbQ_i,q_i,\sigma_i\ra : i <\omega\ra$ of nodes of $T$ converging to $t$, with $\la\bbQ_0,q_0,\sigma_0\ra = \la\bbP,p,\tau\ra$.  If $\bbQ^*$ is the inverse limit of the $\bbQ_i$, and $q^*$ is defined by $q^*(i) = \inf_{i<j<\omega} \pi_{ji}(q_j)$, where $\pi_{ji}$ is the projection chosen by $\Sigma$, and $\sigma^*$ is forced to be the concatenation of the $\sigma_i$, then $\la\bbQ^*,q^*,\sigma^*\ra$ is one lower bound to the sequence.  Note that $q^*(0) = \rho(t) \geq p'$.  Since the restriction map is a projection, we can find another lower bound by strengthening $q^*$ to $q^{**}$ with $q^{**}(0) \leq p'$.  Since $T$ is dense, there is some $\la\bbR,r,\chi\ra \in T$ below it, and if $\pi : \bbR \to \bbP$ is the projection chosen by $\Sigma$, then we must have $\pi(r) \leq p'$, using the fact that $\la\bbP,p,\tau\ra$ is coding.

    The claim now follows, since $\la\bbP,p,\tau\ra$ forces in $\shi(\omega_1,\kappa)$ that $\tau$ is an initial segment of $\tau_X$. 
\end{proof}

Now let us show Theorem~\ref{omega2}.  Suppose $j : V \to M$ is an elementary embedding with critical point $\kappa$, $j(\kappa) = \lambda$, and $M^{<\lambda} \subseteq M$.  By \cite[Theorem 24.11]{kanamori}, we may assume that $j(\lambda)<(\lambda^+)^V$ and $j[\lambda]$ is cofinal in $j(\lambda)$.

Let $G * H \subseteq \shi(\omega_1,\kappa)*\dot\col(\kappa,{<}\lambda)$ be generic over $V$.  We want to show that the embedding can be lifted to one with domain $V[G*H]$ by forcing further with $\col(\omega_1,\kappa)$.  Let $g \subseteq \col(\omega_1,\kappa)$ be generic over $V[G*H]$.  By the above discussion, $\shi(\omega_1,\kappa)\times\col(\omega_1,\kappa)$ is equivalent to a completely $\omega_1$-closed tree $T'$.  The forcing $T' * \dot\col(\kappa,{<}\lambda)^{\shi(\omega_1,\kappa)}$ is a reasonable $(\omega_1,\lambda)$-collapse.  Let $\la\bbR_\alpha : \alpha < \lambda \ra$ be an increasing sequence of regular suborders witnessing reasonableness, with $\bbR_\alpha = T'$ for $\alpha\leq\kappa$.

 Since $M[G][H][g]$ is a $\lambda$-c.c.\ forcing extension of $M$, $M[G][H][g]$ is ${<}\lambda$-closed in $V[G][H][g]$.  Since $j(\lambda)<(\lambda^+)^V$, $\p(\add(\lambda))^{M[G][H][g]}$ has size $\lambda$ in $V[G][H][g]$.  Therefore, we can inductively build in $V[G][H][g]$ a filter $K \subseteq \add(\lambda)$ that is $\add(\lambda)$-generic over $M[G][H][g]$, with $K \restriction\kappa = \tau_X^G$, where $\tau_X$ is the name for the canonical binary sequence of length $\kappa$ added by $\shi(\omega_1,\kappa)$ as above.
 
Let $\psi : T' * \dot\col(\kappa,{<}\lambda)^{\shi(\omega_1,\kappa)} \to \shi(\omega_1,\lambda)$ be as in Lemma~\ref{absorption}.  Then $\la 1,\tau_X \ra \in [(G\times g)*H]*K$, and $\psi(1,\tau_x) = \la T',1,\tau_X\ra$.  Let $G' \subseteq \shi(\omega_1,\lambda)$ be the generic filter generated by $\psi[[(G\times g)*H]*K]$.  Claim~\ref{mc} implies that $G \subseteq G'$.  Therefore, we may lift the embedding to $j : V[G] \to M[G']$.

The quotient forcing to get from $M[G']$ to $M[[(G\times g)*H]*K]$ is $\lambda$-distributive.  Therefore, $M[G']$ is also ${<}\lambda$-closed in $V[G][H][g]$.  For each $\alpha<\lambda$, let $H_\alpha = H \cap \col(\kappa,{<}\alpha)^{V[G]}$.  We have that $j[H_\alpha] \in M[G']$ for each $\alpha<\lambda$.  Since $\col(\lambda,{<}j(\lambda))$ is $\lambda$-closed in $M[G']$, the union of each $j[H_\alpha]$ is a condition $m_\alpha \in \col(\lambda,{<}j(\lambda))^{M[G']}$.  At this stage, we use the same argument as in \cite{eskewdense} and other works to build in $V[G][H][g]$ a filter $H' \supseteq j[H]$ that is $\col(\lambda,{<}j(\lambda))$-generic over $M[G']$.  Namely, we enumerate all the dense open sets that live in $M[G']$ in ordertype $\lambda$ and meet each one  in a way that is compatible with all $m_\alpha$.  This allows us to lift the embedding once again to $j : V[G][H] \to M[G'][H']$.

In $V[G][H]$, we define an ideal $I$ on $\kappa$ as $\{ A \subseteq \kappa : 1 \Vdash_{\col(\omega_1,\kappa)} \kappa \notin j(A) \}$.  It is easy to see that this is a normal ideal on $\kappa$.  We define a map $e : \p(\kappa)/ I \to \calB(\col(\omega_1,\kappa))$ by $e([A]_I) = || \kappa \in j(A) ||$.  The argument for Claim 45 of \cite{eh} shows that $e$ is actually an isomorphism.  This completes the proof.

\bibliographystyle{amsplain.bst}
\bibliography{compare.bib}

\providecommand{\bysame}{\leavevmode\hbox to3em{\hrulefill}\thinspace}
\providecommand{\MR}{\relax\ifhmode\unskip\space\fi MR }
\providecommand{\MRhref}[2]{%
  \href{http://www.ams.org/mathscinet-getitem?mr=#1}{#2}
}
\providecommand{\href}[2]{#2}
\begin{thebibliography}{10}

\bibitem{sargsyan}
Dominik Adolf, Grigor Sargsyan, Nam Trang, Trevor~M. Wilson, and Martin Zeman,
  \emph{Ideals and strong axioms of determinacy}, J. Amer. Math. Soc.
  \textbf{37} (2024), no.~4, 1203--1273. \MR{4777642}

\bibitem{eskewdense}
Monroe Eskew, \emph{Dense ideals and cardinal arithmetic}, J. Symb. Log.
  \textbf{81} (2016), no.~3, 789--813. \MR{3569105}

\bibitem{eh}
Monroe {Eskew} and Yair {Hayut}, \emph{{Dense ideals}}, arXiv e-prints (2024),
  arXiv:2410.14359.

\bibitem{thesis}
Monroe~Blake Eskew, \emph{Measurability {P}roperties on {S}mall {C}ardinals},
  ProQuest LLC, Ann Arbor, MI, 2014, Thesis (Ph.D.)--University of California,
  Irvine. \MR{3279214}

\bibitem{fms}
M.~Foreman, M.~Magidor, and S.~Shelah, \emph{Martin's maximum, saturated
  ideals, and nonregular ultrafilters. {I}}, Ann. of Math. (2) \textbf{127}
  (1988), no.~1, 1--47. \MR{924672}

\bibitem{foremangames}
Matthew Foreman, \emph{Games played on {B}oolean algebras}, J. Symbolic Logic
  \textbf{48} (1983), no.~3, 714--723. \MR{716633}

\bibitem{moresat}
\bysame, \emph{More saturated ideals}, Cabal seminar 79--81, Lecture Notes in
  Math., vol. 1019, Springer, Berlin, 1983, pp.~1--27. \MR{730584}

\bibitem{foremanhandbook}
\bysame, \emph{Ideals and generic elementary embeddings}, Handbook of set
  theory. {V}ols. 1, 2, 3, Springer, Dordrecht, 2010, pp.~885--1147.
  \MR{2768692}

\bibitem{Grigorieff}
Serge Grigorieff, \emph{Intermediate submodels and generic extensions in set
  theory}, Ann. of Math. (2) \textbf{101} (1975), 447--490. \MR{373889}

\bibitem{iy}
Tetsuya Ishiu and Yasuo Yoshinobu, \emph{Directive trees and games on posets},
  Proc. Amer. Math. Soc. \textbf{130} (2002), no.~5, 1477--1485. \MR{1879973}

\bibitem{Jech}
Thomas Jech, \emph{Set theory}, millennium ed., Springer Monographs in
  Mathematics, Springer-Verlag, Berlin, 2003. \MR{1940513}

\bibitem{js}
Ronald Jensen and John Steel, \emph{{$K$} without the measurable}, J. Symbolic
  Logic \textbf{78} (2013), no.~3, 708--734. \MR{3135495}

\bibitem{kanamori}
Akihiro Kanamori, \emph{The higher infinite}, second ed., Springer Monographs
  in Mathematics, Springer-Verlag, Berlin, 2009, Large cardinals in set theory
  from their beginnings, Paperback reprint of the 2003 edition. \MR{2731169}

\bibitem{kunen}
Kenneth Kunen, \emph{Saturated ideals}, J. Symbolic Logic \textbf{43} (1978),
  no.~1, 65--76. \MR{495118}

\bibitem{laver}
Richard Laver, \emph{An {$(\aleph \sb{2},\,\aleph \sb{2},\,\aleph
  \sb{0})$}-saturated ideal on {$\omega \sb{1}$}}, Logic {C}olloquium '80
  ({P}rague, 1980), Stud. Logic Found. Math., vol. 108, North-Holland,
  Amsterdam-New York, 1982, pp.~173--180. \MR{673792}

\bibitem{lietz}
Andreas Lietz, \emph{Forcing ``$\mathrm{NS}_{\omega_1}$ is $\omega_1$-dense''
  from large cardinals}, 2023, Thesis (Ph.D.)-- Westf\"alischen
  Wilhelms-Universit\"at M\"unster.

\bibitem{shelahdense}
Saharon Shelah, \emph{Around classification theory of models}, Lecture Notes in
  Mathematics, vol. 1182, Springer-Verlag, Berlin, 1986. \MR{850051}

\bibitem{woodinshelah}
Saharon Shelah and Hugh Woodin, \emph{Large cardinals imply that every
  reasonably definable set of reals is {L}ebesgue measurable}, Israel J. Math.
  \textbf{70} (1990), no.~3, 381--394. \MR{1074499}

\bibitem{shioyaeaston}
Masahiro Shioya, \emph{Easton collapses and a strongly saturated filter}, Arch.
  Math. Logic \textbf{59} (2020), no.~7-8, 1027--1036. \MR{4159767}

\bibitem{shioyadense}
\bysame, \emph{An explicit collapse for a dense filter}, Israel J. Math.
  \textbf{265} (2025), no.~1, 467--478. \MR{4870839}

\bibitem{woodinbook}
W.~Hugh Woodin, \emph{The axiom of determinacy, forcing axioms, and the
  nonstationary ideal}, revised ed., De Gruyter Series in Logic and its
  Applications, vol.~1, Walter de Gruyter GmbH \& Co. KG, Berlin, 2010.
  \MR{2723878}

\end{thebibliography}
\end{document}